\documentclass[11pt]{amsart}
\usepackage{amsmath,amsthm,amsfonts,amssymb,latexsym,epsfig}

\newtheorem{theorem}{Theorem}[section]

\newtheorem{lemma}{Lemma}[section]

\numberwithin{equation}{section}

\theoremstyle{definition}

\theoremstyle{remark}

\begin{document}
\title{Extensions of Copson's Inequalities}
\author{Peng Gao}
\address{Division of Mathematical Sciences, School of Physical and Mathematical Sciences,
Nanyang Technological University, 637371 Singapore}
\email{penggao@ntu.edu.sg}


\begin{abstract}
  We extend the classical Copson's inequalities so that the values
  of parameters involved go beyond what is currently known.
\end{abstract}

\maketitle
\section{Introduction}
\label{sec 1} \setcounter{equation}{0}

  Let $p>0$ and ${\bf x}=(x_n)_{n \geq 1}$ be a non-negative sequence.
  Let $(\lambda_n)_{n \geq 1}$ be a non-negative sequence with
  $\lambda_1>0$ and let $\Lambda_n=\sum^n_{i=1}\lambda_i$.
   The classical Copson's inequalities are referred as the
   following ones \cite[Theorem 1.1, 2.1]{C}:
\begin{align}
\label{1.1}
  & \sum^{\infty}_{n=1}\lambda_n\Lambda^{-c}_n\left ( \sum^{n}_{k=1}\lambda_k x_k \right
   )^p \leq \left ( \frac {p}{c-1}\right )^p
   \sum^{\infty}_{n=1}\lambda_n\Lambda^{p-c}_nx^p_n, \ 1<c \leq p
   ; \\
  \label{1.2}
  & \sum^{\infty}_{n=1}\lambda_n\Lambda^{-c}_n\left ( \sum^{\infty}_{k=n}\lambda_k x_k \right
   )^p \leq \left ( \frac {p}{1-c}\right )^p
   \sum^{\infty}_{n=1}\lambda_n\Lambda^{p-c}_nx^p_n, \ 0 \leq c < 1.
\end{align}
   When $\lambda_k=1$ for all $k$ and $c=p$, inequality \eqref{1.1} becomes the celebrated
   Hardy's inequality (\cite[Theorem 326]{HLP}). We note that the
   reversed inequality of \eqref{1.2} holds when $c \leq 0 <p <1$
   and the constants are best possible in all these cases.

   It is easy to show that inequalities \eqref{1.1} and \eqref{1.2} are equivalent to each other by the duality principle \cite[Lemma 2]{M}
    for the norms of linear operators. It's observed by Bennett
    \cite[p. 411]{B1} that
   inequality \eqref{1.1} continues to hold for $c>p$ with
   constant $(p/(p-1))^p$. A natural question to ask now is
   whether inequality \eqref{1.1} itself continues to hold for $c>p$. Note
   that in this case the constant $(p/(c-1))^p$ is best possible
   by considering the case $\lambda_n=1$,
   $x_n=n^{(c-p-1-\epsilon)/p}$ with $\epsilon
   \rightarrow 0^+$.

   As analogues to Copson's inequalities, the following
   inequalities are due to Leindler \cite[(1)]{L}:
\begin{align}
\label{1.3}
 & \sum^{\infty}_{n=1}\lambda_n{\Lambda^*_n}^{-c}\left (
\sum^{n}_{k=1}\lambda_k x_k \right
   )^p \leq \left ( \frac {p}{1-c}\right )^p
   \sum^{\infty}_{n=1}\lambda_n{\Lambda^*_n}^{p-c}x^p_n, \ 0 \leq c
   <1
   ; \\
\label{1.4}
  & \sum^{\infty}_{n=1}\lambda_n{\Lambda^*_n}^{-c}\left ( \sum^{\infty}_{k=n}\lambda_k x_k \right
   )^p \leq \left ( \frac {p}{c-1}\right )^p
   \sum^{\infty}_{n=1}\lambda_n{\Lambda^*_n}^{p-c}x^p_n, \ 1 < c \leq
   p,
\end{align}
   where we assume $\sum^{\infty}_{n=1}\lambda_n<\infty$ and we set
   $\Lambda^*_n=\sum^{\infty}_{k=n}\lambda_k$. We point out here
   that Leindler's result corresponds to
   case $c=0$ of inequality \eqref{1.3}, after a change of
   variables. Inequalities \eqref{1.3} and \eqref{1.4} are given in
   \cite[Corollary 5, 6, p. 412]{B1}. Again it is easy to see that
   inequalities \eqref{1.3} and \eqref{1.4} are equivalent to each other by the duality
   principle. Moreover, the constants are best possible.

   As an application of Copson's inequalities, we note the
   following result of Bennett and Grosse-Erdmann \cite[Theorem 8]{BGE1} that asserts for $p \geq 1, \alpha \geq 1$,
\begin{align}
\label{1.5}
  \sum^{\infty}_{n=1}\lambda_n\left (\sum^{\infty}_{k=n}\Lambda^{\alpha }_{k}x_k \right )^n \leq (\alpha
  p+1)^p\sum^{\infty}_{n=1}\lambda_n\Lambda^{\alpha p}_{n}\left
  (\sum^{\infty}_{k=n}x_k \right )^p.
\end{align}
  Here the constant is best possible. They also conjectured \cite[p.
   579]{BGE1} that inequality \eqref{1.5} (resp. its reverse) remains valid with the same best possible
   constant when $p \geq 1, 0 < a < 1$ (resp. $ -1/p < a < 0$).
   Weaker constants are given for these cases in \cite[Theorem 9,
   10]{BGE1}.

   It is our goal in this paper to show in the next section that the method developed
   in \cite{G6}-\cite{G9} can be applied to extend Copson's
   inequality \eqref{1.1} to some $c>p$ (or equivalently, by the duality principle, to extend Copson's
   inequality \eqref{1.2} to some $c<0$). In Section \ref{sec 3}, we
   extend inequality \eqref{1.5} to some $0<\alpha<1$.

\section{Main Result}
\label{sec 2} \setcounter{equation}{0}

   Before we prove our main result, we need a lemma first.
\begin{lemma}
\label{lem1}
   Let $p>0$ be fixed. In order for the following inequality (resp. its
   reverse)
\begin{align}
\label{2.7}
   \frac {1-c}{p}x  \leq \left (1+\frac {1-c}{p}x \right
   )^{1-p}-(1-x)^{1-c}
\end{align}
   to be valid when $c<0, p>1$ (resp. when $0<c<1, 0<p<1$) for all $0 \leq x \leq 1$, it
   suffices that it is valid when $x=1$.
\end{lemma}
\begin{proof}
   As the proofs for both cases are similar, we only consider
   the case $c<0, p>1$ here. Let
\begin{align*}
   f_{p,c}(x)=\left (1+\frac {1-c}{p}x \right
   )^{1-p}-(1-x)^{1-c}-\frac {1-c}{p}x.
\end{align*}
    Note that $f_{p,c}(0)=0$ and we have
\begin{align*}
   f'_{p,c}(x) &=\frac {(1-c)(1-p)}{p}\left (1+\frac {1-c}{p}x \right
   )^{-p}+(1-c)(1-x)^{-c}-\frac {1-c}{p}, \\
   f''_{p,c}(x) &= \frac {(1-c)^2(p-1)}{p}\left (1+\frac {1-c}{p}x \right
   )^{-p-1}+(1-c)c(1-x)^{-c-1}.
\end{align*}
    It is easy to see that $f''_{p,c}(x)=0$ is equivalent to the
    equation $g_{p,c}(x)=0$ where
\begin{align*}
    g_{p,c}(x)=\left ( \frac {-pc}{(p-1)(1-c)}\right
    )^{-1/(p+1)}(1-x)^{(1+c)/(p+1)}-1-\frac {1-c}{p}x.
\end{align*}
    If $0>c> -1$, then it is easy to see that
    $f'''_{p,c}(x)< 0$ so that $f''_{p,c}(x)=0$ has at most one
    root in $(0,1)$. As $\lim_{x \rightarrow 1^-}f''_{p,c}(x)=-\infty$, it follows that if $f''_{p,c}(x) \leq 0$ for all
    $0 \leq x <1$, then $f_{p,c}(x)$ is concave down and the
    assertion of the lemma follows. Otherwise we have $f''_{p,c}(0)>0$ and this combined with the observation that
    $f'_{p,c}(0)=0, f'_{p,c}(1)<0$ implies that there
    exists an $x_0 \in [0,1]$ such that
    $f'_{p,c}(x) \geq 0$ for $0 \leq x \leq x_0$ and $f'_{p,c}(x)
    \leq 0$ for $x_0 \leq x \leq 1$ and the
    assertion of the lemma follows. The case $c=-1$ can be
    similarly discussed.

    If $c <-1$, then $g'_{p,c}(x)=0$ has at most one root in $(0,1)$ so that $f''_{p,c}(x)=0$ has at most
    two roots in $(0,1)$. If $f''_{p,c}(0)<0$, then as $f''_{p,c}(1)>0$, it follows that $f''_{p,c}(x)=0$ has exactly one root in
    $(0,1)$, and as $f'_{p,c}(0)=0, f'_{p,c}(1)<0$, it follows
    that $f'_{p,c}(x) < 0$ for all $x \in [0,1]$ and the
    assertion of the lemma follows. If $f''_{p,c}(0)>0$, then $f''_{p,c}(x)=0$ has either no root or two roots in
    $(0,1)$. If  $f''_{p,c}(x)=0$ has no root in
    $(0,1)$, then $f''_{p,c}(x) \geq 0$ for $x \in [0,1]$. As $f'_{p,c}(0)=0, f'_{p,c}(1) < 0$,
    we see that this is not possible. If $f''_{p,c}(x)=0$ has two roots in
    $(0,1)$, it follows that $f_{p,c}(x)$ is first
    increasing, then decreasing and then increasing again for $x
    \in [0,1]$ and it follows from $f'_{p,c}(1) < 0$ that there
    exists an $x'_0 \in [0,1]$ such that
    $f'_{p,c}(x) \geq 0$ for $0 \leq x \leq x'_0$ and $f'_{p,c}(x)
    \leq 0$ for $x'_0 \leq x \leq 1$ and the
    assertion of the lemma again follows. The case $f''_{p,c}(0)=0$
    can be discussed similarly as above and this completes the
    proof.
\end{proof}

   We now consider extending inequality \eqref{1.2} to $c<0$.
   For two fixed two positive sequences $\{ a_n \}, \{ b_n \}$, we
   recall that it is shown in \cite[Section 6]{G8} that we have
   the following inequality:
\begin{align}
\label{2.01}
    \frac {w^{p-1}_1}{b^{p}_1}\Big (\sum^{\infty}_{k=1}w_k\Big )^{1-p}a^p_1A^p_1+\sum_{n=2}^{N}
    \Big(\sum^{\infty}_{k=n}w_k \Big )^{-(p-1)}\Big( \frac {w_n^{p-1}}{b^p_n}-\frac {w_{n-1}^{p-1}}{b^p_{n-1}} \Big )
    a^p_n A_n^{p} \leq \sum_{n=1}^N
    x_n^{p},
\end{align}
    where $\{ w_n \}$ is a positive sequence, $N$ is a large integer and for $1 \leq n \leq N$, we set
    $S_n=\sum^{N}_{k=n}b_kx_k$ and $A_n=S_n/a_n$.

    We now recast inequality \eqref{1.2} as
\begin{align}
\label{2.02}
   \sum^{\infty}_{n=1}\left ( \lambda^{1/p}_n\Lambda^{-c/p}_n\sum^{\infty}_{k=n}\lambda^{1-1/p}_k \Lambda^{-(1-c/p)}_kx_k \right
   )^p \leq \left ( \frac {p}{1-c}\right )^p
   \sum^{\infty}_{n=1}x^p_n.
\end{align}

    It remains to establish inequality \eqref{2.02}. For this, it suffices to establish inequality \eqref{2.02} with
    the infinite sums replaced by finite sums from $1$ to $N$.
 We may also assume $\lambda_n>0$ for all $n$. We then set
\begin{align*}
   a_n=\lambda^{-1/p}_n\Lambda^{c/p}_n, \ b_n=\lambda^{1-1/p}_n
   \Lambda^{-(1-c/p)}_n
\end{align*}
    in inequality \eqref{2.01} to see that in order to establish inequality
    \eqref{2.02}, it suffices to find a positive sequence $\{
    w_n \}$ such that
\begin{align*}
   \Big( \sum^{\infty}_{k=n}w_k \Big )^{p-1} & \leq \left ( \frac {p}{1-c}\right )^p\lambda^{-1}_n\Lambda^{c}_n
   \Big( \frac {w_n^{p-1}\Lambda^{p-c}_n}{\lambda^{p-1}_n}-\frac {w_{n-1}^{p-1}\Lambda^{p-c}_{n-1}}{\lambda^{p-1}_{n-1}} \Big
   ),  \ n \geq 2;  \\
   \Big( \sum^{\infty}_{k=1}w_k \Big )^{p-1} & \leq \left ( \frac {p}{1-c}\right
   )^p\lambda^{-1}_1\Lambda^{c}_1
   \frac {w_1^{p-1}\Lambda^{p-c}_1}{\lambda^{p-1}_1}=\left ( \frac {p}{1-c}\right
   )^p
   \frac {w_1^{p-1}\Lambda^{p-1}_1}{\lambda^{p-1}_1}.
\end{align*}
   Upon a change of variables: $w_n \rightarrow \lambda_nw_n$, we
   can recast the above inequalities as
\begin{align}
\label{2.03}
   \left( \frac {1}{\Lambda_n}\sum^{\infty}_{k=n}\lambda_kw_k \right )^{p-1} & \leq \left ( \frac {p}{1-c}\right )^p\frac {\Lambda_n}{\lambda_n}
   \left( w_n^{p-1}-w_{n-1}^{p-1}\left (\frac {\Lambda_{n-1}}{\Lambda_{n}} \right )^{p-c}
   \right
   ),  \ n \geq 2;  \\
\label{2.04}
   \left( \frac {1}{\Lambda_1} \sum^{\infty}_{k=1}\lambda_kw_k \right )^{p-1} & \leq \left ( \frac {p}{1-c}\right
   )^pw_1^{p-1}.
\end{align}

   We now define the sequence $\{ w_n \}$ inductively by setting $w_1=1$ and for $n \geq 2$,
\begin{align*}
   \sum^{\infty}_{k=n}\lambda_kw_k= \frac
   {p}{1-c}\Lambda_{n-1}w_{n-1}.
\end{align*}
    This implies that for $n \geq 2$,
\begin{align*}
    w_{n}=\frac {\Lambda_{n-1}}{\Lambda_n}\left (1+\frac {1-c}{p}\frac  {\lambda_n}{\Lambda_n}\right )^{-1}w_{n-1}.
\end{align*}
    Using the above relations, we can simplify inequalities
    \eqref{2.03}, \eqref{2.04} to see that inequality \eqref{2.03} is equivalent to
    inequality \eqref{2.7} with $x=\lambda_n/\Lambda_n$ while inequality \eqref{2.04} is equivalent to
\begin{align}
\label{2.6}
    \left (1+\frac {1-c}{p} \right
   )^{1-p}-\frac {1-c}{p} \geq 0.
\end{align}
   It is easy to see that the above inequality is just the case $x=1$ of inequality \eqref{2.7}, we
   then conclude from Lemma \ref{lem1} that inequality \eqref{1.2} is valid for $c<0$ as
   long as the above inequality holds.


   Next, we consider extending inequality \eqref{1.3} to $c<0$. For two fixed two positive sequences $\{ a_n \}, \{ b_n \}$, we
recall that it is shown in \cite[(3.6)]{G7} (see also the
discussion in Section 5 of \cite{G9}) that in order for the
   following inequality
\begin{align*}
   \sum^{\infty}_{n=1}\left ( \sum^{n}_{k=1}a_nb_kx_k \right )^p
   \leq U_p \sum^{\infty}_{n=1}x^p_n.
\end{align*}
   to be valid for a given constant $U_p, p>1$, it suffices to find a positive sequence
   $\{w_n \}$ such that
\begin{align}
\label{2.1}
   \Big(\sum_{k=1}^nw_k \Big )^{p-1} \leq U_pa^p_n \Big ( \frac {w_n^{p-1}}{b^p_n}-\frac {w_{n+1}^{p-1}}{b^p_{n+1}} \Big
   ).
\end{align}

   Without loss of generality, we may assume $\lambda_n>0$ for all $n$. By a change of variables, we recast inequality \eqref{1.3} as
\begin{align*}
   \sum^{\infty}_{n=1}\left (\lambda^{1/p}_n{\Lambda^*}^{-c/p}_n \sum^{n}_{k=1}\lambda^{1-1/p}_k{\Lambda^*}^{-(1-c/p)}_k x_k \right
   )^p \leq \left ( \frac {p}{1-c}\right )^p
   \sum^{\infty}_{n=1}x^p_n.
\end{align*}
   It follows from \eqref{2.1} that in order to establish the above inequality, it suffices
   to
   find a positive sequence $\{w_n \}$ such that
\begin{align*}
   \Big(\sum_{k=1}^nw_k \Big )^{p-1} \leq \left ( \frac {p}{1-c}\right )^p\frac {{\Lambda^*}^c_n}{\lambda_n}
   \Big ( \frac {w_n^{p-1}{\Lambda^*}^{p-c}_n}{\lambda^{p-1}_n}-\frac {w_{n+1}^{p-1}{\Lambda^*}^{p-c}_{n+1}}{\lambda^{p-1}_{n+1}} \Big
   ).
\end{align*}
   By a change of variables: $w_n \mapsto \lambda_n w_n$, we can
   recast the above inequality as
\begin{align}
\label{2.2}
   \left(\frac {1}{\Lambda^*_n}\sum_{k=1}^n\lambda_kw_k \right )^{p-1} \leq \left ( \frac {p}{1-c}\right )^p\frac {\Lambda^*_n}{\lambda_n}
   \left ( w_n^{p-1}-w_{n+1}^{p-1}\left (\frac {\Lambda^*_{n+1}}{\Lambda^*_{n}} \right )^{p-c}
   \right
   ).
\end{align}
   We now define the sequence  $\{w_n \}$  inductively by setting
   $w_1=1$ and for $n \geq 1$,
\begin{align*}
   \sum_{k=1}^n\lambda_kw_k=\frac
   {p}{1-c}\Lambda^*_{n+1}w_{n+1}.
\end{align*}
    This implies that for $n \geq 2$,
\begin{align*}
    w_{n}=\frac {\Lambda^*_{n+1}}{\Lambda^*_n}\left (1+\frac {1-c}{p}\frac  {\lambda_n}{\Lambda^*_n}\right )^{-1}w_{n+1}.
\end{align*}
    Using the above relations, we can simplify inequality
    \eqref{2.2} to see that the $n \geq 2$ cases are equivalent to
    inequality \eqref{2.7} with $x=\lambda_n/\Lambda^*_n$.
   It is also easy to see that the $n=1$ case of \eqref{2.2}
   corresponds to the following inequality:
\begin{align*}
   \frac {1-c}{p}x  \leq \left (\frac {1-c}{p}x \right
   )^{1-p}-(1-x)^{1-c}.
\end{align*}
   It is easy to see that the above inequality is implied by inequality \eqref{2.7}, we
   then conclude from Lemma \ref{lem1} that inequality \eqref{1.3} holds for $c<0$ as long as inequality \eqref{2.6} holds.


    Note that for fixed $p>0$, the function $(1+x)^{1-p}-x$ is a
 decreasing function of $x$. Moreover, it is easy to see that
    inequality \eqref{2.6} (resp. its reverse) always holds with $c=0$ when $p>1$ (resp. when $0<p<1$).
    We note that our discussions above for inequality \eqref{1.2} can be carried out for the case $0<p<1, 0<c<1$ with the related inequalities
    reversed. We therefore obtain the following
\begin{theorem}
\label{thm1} Let $p>0$ be fixed. Let $c_0$ denote the unique
number satisfying
\begin{align*}
    \left (1+\frac {1-c_0}{p} \right
   )^{1-p}-\frac {1-c_0}{p} =0.
\end{align*}
   Then inequalities \eqref{1.2} and \eqref{1.3} hold for all
   $c_0 \leq c <1$ when $p>1$ and the reversed inequality \eqref{1.2} holds for all
   $c <c_0$ when $0<p<1$.
\end{theorem}

   We leave it to the reader for the corresponding extensions to
   $c>p$ of inequalities \eqref{1.1} and \eqref{1.4} by the
   duality principle.

\section{Some related results}
\label{sec 3} \setcounter{equation}{0}
    In this section we first consider the conjecture of Bennett and Grosse-Erdmann on inequality \eqref{1.5} for the case
    $0<\alpha<1$. We may assume $\lambda_n>0$ for all $n$. We note here that it is shown in \cite[(153), (156)]{BGE1} that it
    suffices to show that
 \begin{align*}
     \sum^{\infty}_{n=1}\lambda_n \left( \sum^{\infty}_{k=n}
     \left (\Lambda^{\alpha}_k-\Lambda^{\alpha}_{k-1} \right )x_k \right )^p
     \leq (\alpha
     p)^p\sum^{\infty}_{n=1}\lambda_n(\Lambda^{\alpha}_nx_n)^p,
 \end{align*}
   where we set $\Lambda_0=0$. By the duality principle, it is
   easy to see that the above inequality is equivalent to
 \begin{align*}
   \sum^{\infty}_{n=1}\left( \frac {\Lambda^{\alpha}_n-\Lambda^{\alpha}_{n-1}}{\lambda^{1-1/p}_n\Lambda^{\alpha}_n}\sum^{n}_{k=1}
   \lambda^{1-1/p}_kx_k  \right )^p
     \leq \left( \frac {\alpha p}{p-1} \right
     )^p\sum^{\infty}_{n=1}x^p_n.
 \end{align*}

    It follows from \eqref{2.1} that in order to establish the above inequality, it suffices
   to find a positive sequence $\{w_n \}$ such that
\begin{align*}
   \Big(\sum_{k=1}^nw_k \Big )^{p-1} \leq \left ( \frac {\alpha p}{p-1}\right )^p
   \left( \frac
   {\Lambda^{\alpha}_n-\Lambda^{\alpha}_{n-1}}{\lambda^{1-1/p}_n\Lambda^{\alpha}_n}\right
   )^{-p}
   \Big ( \frac {w_n^{p-1}}{\lambda^{p-1}_n}-\frac {w_{n+1}^{p-1}}{\lambda^{p-1}_{n+1}} \Big
   ).
\end{align*}
   By a change of variables: $w_n \mapsto \lambda_n w_n$, we can
   recast the above inequality as
\begin{align}
\label{3.1}
   \left(\frac {1}{\Lambda_n}\sum_{k=1}^n\lambda_kw_k \right )^{p-1} \leq \left ( \frac {p}{p-1}\right
   )^p
   \left( \frac
   {\alpha \lambda_n\Lambda^{\alpha-1}_n}{\Lambda^{\alpha}_n-\Lambda^{\alpha}_{n-1}}\right
   )^{p}\frac {\Lambda_n}{\lambda_n}
   \left ( w_n^{p-1}-w_{n+1}^{p-1}
   \right
   ).
\end{align}
   We now define the sequence $\{w_n \}$ inductively by setting
   $w_1=1$ and for $n \geq 1$,
\begin{align*}
    \sum_{k=1}^n\lambda_kw_k=\frac
   {p}{p-1}\Lambda_nw_{n+1}.
\end{align*}
    This implies that
\begin{align*}
    w_{n+1}=\left (1-\frac {1}{p} \frac {\lambda_n}{\Lambda_n}\right ) w_n.
\end{align*}
    Using the above relations, we can simplify inequality
    \eqref{3.1} to see that it is equivalent to the following:
\begin{align}
\label{3.8}
   \left (\frac {p}{p-1} \right )\left (\left(1-\frac {x}{p} \right )^{1-p}-1 \right )
   \geq x\left ( \frac {1-(1-x)^{\alpha}}{\alpha x} \right )^{p},
\end{align}
   where we set $x=\lambda_n/\Lambda_n$ so that $0 \leq x \leq 1$.

   By Hadamard's inequality, which asserts for a continuous convex function $h(u)$ on $[a, b]$,
\begin{equation*}
    \frac {1}{b-a}\int^b_ah(u)du \geq h(\frac {a+b}{2}),
\end{equation*}
   we see that
\begin{align*}
   \left (\frac {p/x}{p-1} \right )\left (\left(1-\frac {x}{p} \right )^{1-p}-1 \right )
   =\frac {1}{1-(1-x/p)} \int^1_{1-x/p}u^{-p}du \geq \left
   (1-\frac {x}{2p}
   \right )^{-p} .
\end{align*}

   Thus, it remains to show that
\begin{align*}
   \left (1-\frac {x}{2p}
   \right )^{-1}
   \geq \frac {1-(1-x)^{\alpha}}{\alpha x},
\end{align*}
   Equivalently, we need to show $f_{\alpha,p}(x) \geq 0$ where
\begin{align*}
   f_{\alpha,p}(x)=\alpha x-\left (1-\frac {x}{2p} \right )\left (1-(1-x)^{\alpha} \right ).
\end{align*}
   It's easy to see that $f_{\alpha, p}(0)=f'_{\alpha,p}(0)=0$ and $f''_{\alpha,p}(x)$ has a most one root
   in $(0,1)$. It follows that $f'_{\alpha,p}(x)$ has a most one root
   in $(0,1)$. Suppose $\alpha>1-1/p$ so that $f''_{\alpha, p}(0) >0$. This together with the observation that $\lim_{x \rightarrow
   1^-}f'_{\alpha,p}(x)=-\infty$ implies that in order for
   $f_{\alpha,p}(x) \geq 0$ for all $x \in [0,1]$, it suffices to
   have $f_{\alpha,p}(1) \geq 0$. We then deduce that we need to
   have
\begin{align*}
   \alpha \geq 1-\frac {1}{2p}.
\end{align*}
   We then obtain the following
\begin{theorem}
\label{thm2} Inequality \eqref{1.5} is valid for $p>1, \alpha \geq
1-\frac {1}{2p}$.
\end{theorem}

   We now consider the following analogue to inequality
   \eqref{1.5}:
\begin{align}
\label{3.2}
    \sum^{\infty}_{n=1}\lambda_n \left( \sum^{n}_{k=1}
     {\Lambda^*}^{\alpha}_kx_k \right )^p
     \leq (\alpha
     p+1)^p\sum^{\infty}_{n=1}\lambda_n{\Lambda^*}^{\alpha p}_n\left (\sum^{n}_{k=1}x_n \right
     )^p.
\end{align}

    Again we may assume $\lambda_n>0$ for all $n$. We set
\begin{align*}
    y_n=\sum^n_{k=1}x_k
\end{align*}
    to recast inequality \eqref{3.2} as
\begin{align*}
    \sum^{\infty}_{n=1}\lambda_n \left( \sum^{n-1}_{k=1}
     \left ({\Lambda^*}^{\alpha}_k-{\Lambda^*}^{\alpha}_{k+1} \right )y_k+ {\Lambda^*}^{\alpha}_{n} y_n \right )^p
     \leq (\alpha
     p+1)^p\sum^{\infty}_{n=1}\lambda_n{\Lambda^*}^{\alpha p}_ny^p_n.
\end{align*}
    By Minkowski's inequality, we have
\begin{align*}
   & \left ( \sum^{\infty}_{n=1}\lambda_n \left( \sum^{n-1}_{k=1}
     \left ({\Lambda^*}^{\alpha}_k-{\Lambda^*}^{\alpha}_{k+1} \right )y_k+ {\Lambda^*}^{\alpha}_{n} y_n \right
     )^p \right )^{\frac {1}{p}} \\
   \leq & \left ( \sum^{\infty}_{n=1}\lambda_n \left( \sum^{n-1}_{k=1}
     \left ({\Lambda^*}^{\alpha}_k-{\Lambda^*}^{\alpha}_{k+1} \right )y_k \right
     )^p \right )^{\frac {1}{p}}+\left ( \sum^{\infty}_{n=1}\lambda_n \left( {\Lambda^*}^{\alpha}_{n} y_n \right
     )^p \right )^{\frac {1}{p}}.
\end{align*}
    Thus, it suffices to show that
\begin{align}
\label{3.4}
    \sum^{\infty}_{n=1}\lambda_n \left( \sum^{n}_{k=1}
     \left ({\Lambda^*}^{\alpha}_k-{\Lambda^*}^{\alpha}_{k+1} \right )y_k \right
     )^p \leq ( \alpha
     p)^p\sum^{\infty}_{n=1}\lambda_n{\Lambda^*}^{\alpha p}_ny^p_n.
\end{align}

    When $0 < \alpha \leq 1$, we note that we have
\begin{align*}
    \sum^{\infty}_{n=1}\lambda_n \left( \sum^{n}_{k=1}
     \left ({\Lambda^*}^{\alpha}_k-{\Lambda^*}^{\alpha}_{k+1} \right )y_k \right
     )^p \leq \sum^{\infty}_{n=1}\lambda_n \left( \sum^{n}_{k=1}
     \left (\alpha \lambda_k {\Lambda^*}^{\alpha-1}_k\right )y_k \right
     )^p.
\end{align*}
    It then follows from inequality \eqref{1.3} with $c=0,
    x_k={\Lambda^*}^{\alpha-1}_k y_k$ that
\begin{align*}
    \sum^{\infty}_{n=1}\lambda_n \left( \sum^{n}_{k=1}
     \lambda_k{\Lambda^*}^{\alpha-1}_ky_k \right
     )^p \leq p^p\sum^{\infty}_{n=1}\lambda_n{\Lambda^*}^{\alpha p}_ny^p_n.
\end{align*}
     Thus, inequality \eqref{3.4} is valid when $0<\alpha \leq 1$.

     We now consider the case $\alpha \geq 1$. By the duality principle, it is
   easy to see that inequality \eqref{3.4} is equivalent to
 \begin{align}
 \label{3.5}
   \sum^{\infty}_{n=1}\left( \frac {{\Lambda^*}^{\alpha}_n-{\Lambda^*}^{\alpha}_{n+1}}{\lambda^{1-1/p}_n{\Lambda^*}^{\alpha}_n}\sum^{\infty}_{k=n}
   \lambda^{1-1/p}_ky_k  \right )^p
     \leq \left( \frac {\alpha p}{p-1} \right
     )^p\sum^{\infty}_{n=1}y^p_n.
 \end{align}

    We then see that upon setting
\begin{align*}
   a_n=\frac {\lambda^{1-1/p}_n{\Lambda^*}^{\alpha}_n}{{\Lambda^*}^{\alpha}_n-{\Lambda^*}^{\alpha}_{n+1}}, \ b_n=\lambda^{1-1/p}_n
\end{align*}
    in inequality \eqref{2.01} that one can establish inequality
    \eqref{3.5} as long as one can find a positive sequence $\{
    w_n \}$ such that
\begin{align*}
   \Big( \sum^{\infty}_{k=n}w_k \Big )^{p-1} & \leq \left ( \frac {\alpha p}{p-1}\right )^p
   \left ( {\Lambda^*}^{\alpha}_n-{\Lambda^*}^{\alpha}_{n+1} \right )^{-p}\lambda^{p-1}_n{\Lambda^*}^{\alpha p}_n
   \Big( \frac {w_n^{p-1}}{\lambda^{p-1}_n}-\frac {w_{n-1}^{p-1}}{\lambda^{p-1}_{n-1}} \Big
   ),  \ n \geq 2;  \\
   \Big( \sum^{\infty}_{k=1}w_k \Big )^{p-1} & \leq \left ( \frac {\alpha p}{p-1}\right
   )^p\left ( {\Lambda^*}^{\alpha}_1-{\Lambda^*}^{\alpha}_{2} \right )^{-p}\lambda^{p-1}_1{\Lambda^*}^{\alpha
   p}_1
   \frac {w_1^{p-1}}{\lambda^{p-1}_1}.
\end{align*}
   Upon a change of variables: $w_n \rightarrow \lambda_nw_n$, we
   can recast the above inequalities as
\begin{align}
\label{3.6}
   \Big( \frac {1}{\Lambda^*_n}\sum^{\infty}_{k=n}\lambda_kw_k \Big )^{p-1} & \leq \left ( \frac {p}{p-1}\right )^p
   \left ( \frac {\alpha \lambda_n {\Lambda^*}^{\alpha-1}_n}{{\Lambda^*}^{\alpha}_n-{\Lambda^*}^{\alpha}_{n+1}} \right )^{p}
   \frac {{\Lambda^*}_n}{\lambda_n}
   \Big( w_n^{p-1}-w_{n-1}^{p-1}\Big
   ),  \ n \geq 2;  \\
\label{3.7}
   \Big( \frac {1}{\Lambda^*_1}\sum^{\infty}_{k=1}\lambda_k w_k \Big )^{p-1} & \leq \left ( \frac {p}{p-1}\right
   )^p\left ( \frac {\alpha \lambda_1 {\Lambda^*}^{\alpha-1}_1}{{\Lambda^*}^{\alpha}_1-{\Lambda^*}^{\alpha}_{2}} \right )^{p}
   \frac {{\Lambda^*}_1}{\lambda_1}
   w_1^{p-1}.
\end{align}

   We now define the sequence $\{w_n \}$ inductively by setting
   $w_1=1$ and for $n \geq 2$,
\begin{align*}
  \sum_{k=n}^{\infty}\lambda_kw_k=\frac
   {p}{p-1} \Lambda^*_{n}w_{n-1}.
\end{align*}
    This implies that for $n \geq 2$,
\begin{align*}
    w_{n}=\left (1-\frac {1}{p} \frac {\lambda_n}{\Lambda^*_n}\right )^{-1} w_{n-1}.
\end{align*}
    Using the above relations, we can simplify inequality
    \eqref{3.6} to see that it is equivalent to inequality
    \eqref{3.8} with $x=\lambda_n/\Lambda^*_n$ while inequality
    \eqref{3.7} is equivalent to the following inequality
\begin{align*}
   \left (\frac {p}{p-1} \right )\left (1-\frac {x}{p}  \right )^{1-p}
   \geq x\left ( \frac {1-(1-x)^{\alpha}}{\alpha x} \right )^{p}.
\end{align*}
   As  the above inequality is implied by inequality \eqref{3.8}, it suffices to establish inequality \eqref{3.8} for all $\alpha \geq 1$.
   For this, we note that it is easy to show that the right-hand side expression of \eqref{3.8} is a decreasing function of $\alpha$ and
    inequality \eqref{3.8} is valid when $\alpha=1$. It therefore
    follows that inequality \eqref{3.8} is valid for all $\alpha
    \geq 1$. As it is easy to check that the constant in \eqref{3.2} is best possible by considering
    $\lambda_n=n^{-a}, a>1, x_k=n^b, b=((a-1)(\alpha
    p+1)-\epsilon)/p-1$ with $\epsilon \rightarrow 0^+$,
    we conclude the paper with the following
\begin{theorem}
\label{thm3} Inequality \eqref{3.2} is valid for $p>1, \alpha >0$.
The constant is best possible.
\end{theorem}

\end{document}